\documentclass[12pt]{amsart}

\usepackage{amsmath}
\usepackage{amssymb}
\usepackage{amsthm}
\usepackage[noadjust]{cite}
\usepackage{mathrsfs}
\usepackage{dsfont}
\usepackage{dcpic}
\usepackage{hyperref}
\usepackage[backgroundcolor=green!40, linecolor=green!40]{todonotes}
\usepackage[all,cmtip]{xy}
\usepackage{xcolor}
\usepackage{yfonts}
\usepackage{enumerate}
\usepackage{bbm}
\usepackage{dsfont}
\usepackage{verbatim}
\usepackage{setspace}
\usepackage{mathtools}
\numberwithin{equation}{subsection}

\usepackage[top=1in, bottom=1.25in, left=1.25in, right=1.25in]{geometry}

\hypersetup{colorlinks=true,linkcolor=red,citecolor=blue}

\newtheorem{lemma}{Lemma}[section]
\newtheorem{proposition}[lemma]{Proposition}
\newtheorem{theorem}[lemma]{Theorem}
\newtheorem{corollary}[lemma]{Corollary}

\newtheorem*{theorem2}{Theorem}

\theoremstyle{definition}
\newtheorem{definition}[lemma]{Definition}
\newtheorem{remark}[lemma]{Remark}

\def\C{{\mathbb C}}

\def\F{{\mathbb F}}

\def\Z{{\mathbb Z}}

\def\Aa{\mathcal{A}}

\def\Cc{\mathcal{C}}

\def\Gg{\mathcal{G}}

\def\Ii{\mathcal{I}}

\def\Ll{\mathcal{L}}

\def\AA{\mathfrak{A}}
\def\BB{\mathfrak{B}}
\def\CC{\mathfrak{C}}

\def\HH{\mathfrak{H}}

\def\JJ{\mathfrak{J}}
\def\KK{\mathfrak{K}}

\def\PP{\mathfrak{P}}
\def\QQ{\mathfrak{Q}}

\def\SS{\mathfrak{S}}

\def\GGG{\mathbf{G}}

\def\KKK{\mathbf{K}}

\def\NNN{\mathbf{N}}

\def\XXX{\mathbf{X}}

\def\kkk{\mathbf{k}}

\def\ker{\operatorname{ker}}
\def\Hom{\operatorname{Hom}}
\def\End{\operatorname{End}}
\def\Aut{\operatorname{Aut}}
\def\Ind{\operatorname{Ind}}
\def\cInd{c\text{-}\operatorname{Ind}}
\def\Res{\operatorname{Res}}
\def\Gal{\operatorname{Gal}}
\def\Irr{\mathrm{Irr}}
\def\Mat{\mathrm{Mat}}
\def\GL2{\mathbf{GL}_2}
\def\SL2{\mathbf{SL}_2}
\def\GLN{\mathbf{GL}_N}
\def\SLN{\mathbf{SL}_N}

\def\Rep{\mathrm{Rep}}
\def\rec{\mathrm{rec}}
\def\iner{\mathrm{iner}}

\def\sc{\mathrm{sc}}

\newcommand{\dedekind}{{\scriptstyle\mathcal{O}}}

\def\et{\mathrm{et}}

\makeatletter
\renewcommand\subsection{\@startsection {subsection}{1}{\z@}%
                                        {-3.5ex \@plus -1ex \@minus -.2ex}%
                                        {2.3ex \@plus.2ex}%
                                        {\normalfont\textit}}
\makeatother
\begin{document}
\newtheorem*{recall}{Recall}
\newtheorem*{notat}{Notations}
\newtheorem*{problem}{Problem}
\newtheorem*{fact}{Fact}

\title{On the unicity of types in special linear groups}
\author{Peter Latham}
\address{Peter Latham\newline Department of Mathematics\newline King's College\newline Strand\newline London\newline WC2R 2LS\newline United Kingdom}
\email{peter.latham@kcl.ac.uk}

\begin{abstract}
Let $F$ be a non-archimedean local field. We show that any representation of a maximal compact subgroup of $\SLN(F)$ which is typical for an essentially tame supercuspidal representation must be induced from a Bushnell--Kutzko maximal simple type. From this, we explicitly count and describe the conjugacy classes of such typical representations, and give an explicit description of an inertial Langlands correspondence for essentially tame irreducible $N$-dimensional projective representations of the Weil group of $F$.
\end{abstract}
\maketitle
\setlength{\parindent}{0pt}

\section{Introduction}
Let $\GGG$ be a connected reductive group defined over a non-archimedean local field $F$ with ring of integers $\dedekind$, and let $G=\GGG(F)$. Given a supercuspidal representation $\pi$ of $G$, we say that a \emph{type} for $\pi$ is a pair $(J,\lambda)$ consisting of an irreducible representation $\lambda$ of a compact open subgroup $J$ of $G$ such that the only irreducible representations of $G$ which contain $\lambda$ upon restriction to $J$ are the twists of $\pi$ by an unramified character of $G$.\\

In many cases, including those of $\GGG=\GLN$ and $\GGG=\SLN$ with which this paper will be concerned, it is known that there exists a type for every supercuspidal representation of $G$ \cite{bushnell1993admissible,bushnell1993sln,bushnell1994sln}; this construction of types is completely explicit, and results in a unique conjugacy class of \emph{maximal simple types} which are contained in $\pi$. These maximal simple types are defined from \emph{strata} (a very specific equivalence class of such strata; see section \ref{BKsection}), which are essentially the data of a hereditary $\dedekind$-order $\AA$ in $\Mat_N(F)$ and an algebraic extension $E/F$ of degree dividing $N$. The order $\AA$ has a lattice period $e_\AA$, which coincides with the ramification degree of the extension $E/F$. For a supercuspidal representation $\pi$ of $\GLN(F)$ or $\SLN(F)$, we denote by $e_{\pi}$ the lattice period of the associated hereditary order.\\

In this paper, we complete the classification of types for a large class of supercuspidal representations of $\SLN(F)$ -- those which are \emph{essentially tame}, which is to say those supercuspidal representations $\pi$ for which $e_\pi$ is coprime to $N$. We show that the only types for such a representation are the maximal simple types, together with those types obtained from simple representation theoretic renormalizations of maximal simple types.

\begin{theorem2}
Let $\pi$ be an essentially tame supercuspidal representation of $\SLN(F)$. Then the number of $\SLN(F)$-conjugacy classes of types $(K,\tau)$ for $\pi$ with $K\subset\SLN(F)$ a maximal compact subgroup is precisely $e_\pi$, and any two such types for $\pi$ are conjugate by an element of $\GLN(F)$. Each of these types is of the form $\tau=\Ind_J^K\ \mu$ for $(J,\mu)$ a maximal simple type contained in $\pi$.
\end{theorem2}

This generalizes the previous result of Paskunas \cite{paskunas2005unicity} (which is due to Henniart for $N=2$ \cite{breuil2002multiplicites}) that any supercuspidal representation $\pi$ of $\GLN(F)$ contains a unique conjugacy class of types defined on maximal compact subgroups, as well as subsuming a previous result of the author for $N=2$ and $F$ of odd residual characteristic \cite{latham2015sl2}. We note that while it should be expected that the result is true without the assumption that $\pi$ is essentially tame, there are some serious arithmetic difficulties which arise if one drops this assumption (namely, for non-essentially tame supercuspidals it is possible for a maximal simple type to intertwine with its twist by some character of large level; our method of proof seems to be poorly suited to dealing with this problem).\\

We also give an application of this result, which explicitly describes an inertial form of the local Langlands correspondence for essentially tame projective Galois representations. Let $I_F\subset W_F$ denote the inertia and Weil group of some separable algebraic closure $\bar{F}/F$.

\begin{theorem2}
There exists a canonical surjective, finite-to-one map $\iner$ from the set of $\SLN(F)$-conjugacy classes of types $(K,\tau)$ for essentially tame supercuspidal representations of $\SLN(F)$ with $K\subset\SLN(F)$ maximal compact, and the set of equivalence classes of $N$-dimensional projective representations of $I_F$ of $F$ which extend to an essentially tame $L$-parameter for $\SLN(F)$.\\

Given an essentially tame irreducible projective representation $\varphi:W_F\rightarrow\mathbf{PGL}_N(\C)$, let $\Pi$ be the $L$-packet of supercuspidal representations of $\SLN(F)$ associated to $\varphi$, and let $\pi\in\Pi$. Then the fibre of $\iner$ above $\varphi|_{I_F}$ is of cardinality $e_\pi\cdot|\Pi|$.
\end{theorem2}

\subsection{Acknowledgements}
The work contained in this paper was supported by an EPSRC studentship as well the Heilbronn Institute for Mathematical Research, and is based on a part of my UEA PhD thesis; I would like to thank Shaun Stevens for his supervision. I am also grateful to Colin Bushnell for pointing out a mistake in a previous draft of the paper, and to Maarten Solleveld for a number of helpful comments.

\section{Notation}

Let $F$ be a non-archimedean local field with ring of integers $\dedekind=\dedekind_F$, maximal ideal $\frak{p}=\frak{p}_E$ and residue field $\kkk=\kkk_E$ of cardinality $q_F$ and characteristic $p$. We write $G$ for the group $\GLN(F)$ and $\bar{G}=\SLN(F)$. Given $H\subset G$ a closed subgroup, we let $\bar{H}=H\cap\bar{G}$.\\

All conjugacies taking place in the paper will be from the left action; for $x\in H\subset G$ and $g\in G$, we write $^gx=gxg^{-1}$, and given a representation $\sigma$ of $H$, we let $^g\sigma$ be the representation of $^gH$ which acts as $^g\sigma({}^gx)=\sigma(x)$.\\

All representations under consideration will be defined over the complex numbers. For a group $H$, we denote by $\Rep(H)$ the category of smooth representations of $H$, and by $\Irr(H)$ the set of isomorphism classes of irreducible objects in $\Rep(H)$. Any representations we consider will be assumed to be smooth.\\

We denote by $\XXX(F)$ the group of complex characters  $\chi:F^\times\rightarrow\C^\times$, and also fix notation for two subgroups of this. We write $\XXX_{\mathrm{nr}}(F)$ for the subgroup of characters $\chi$ which are unramified, i.e. for which $\chi|_{\dedekind^\times}$ is trivial, and $\XXX_N(F)$ for the subgroup of characters $\chi$ for which $\chi^N$ is unramified.\\

Given subgroups $J,J'$ of $G$ and irreducible representations $\lambda,\lambda'$ of $J$ and $J'$, respectively, we denote the intertwining of $(J,\lambda)$ and $(J',\lambda')$ by $I_G(\lambda,\lambda')=\{g\in G\ | \ \Hom_{J\cap{}^gJ'}(\lambda,{}^g\lambda')\neq 0\}$.

\section{The Bushnell--Kutzko theory}\label{BKsection}

We begin by recalling the necessary background on the theory of types, which will underlie all of the work in this paper. We make no attempt to be comprehensive; the reader should consult \cite{bushnell1993admissible,bushnell1993sln,bushnell1994sln} for a complete account.

\subsection{Strata}

Let $V$ be an $N$-dimensional $F$-vector space, and let $A=\End_F(V)$. Then $A^\times=\Aut_F(V)\simeq G$. We also fix, once and for all, a level 1 additive character $\psi$ of $F$, i.e. a character trivial on $\frak{p}$ but not on $\dedekind$.\\

A \emph{hereditary $\dedekind$-order} in $A$ is an $\dedekind$-order $\AA$ such that every left $\AA$-lattice is $\AA$-projective. Given such an order $\AA$, let $\PP=\PP_\AA$ denote its Jacobson radical; thus $\PP$ is a two-sided invertible fractional ideal of $\AA$, and there exists a unique integer $e_\AA=e_{\AA/\dedekind_F}$ called the \emph{lattice period} of $\AA$ such that $\varpi\AA=\PP^{e_\AA}$.\\

To a hereditary order $\AA$, we associate a number of subgroups. Firstly, let $\KK_\AA=\{x\in G\ | ^x\AA=\AA\}$, which we call the \emph{normalizer} of $\AA$. This is an open, compact-modulo-centre subgroup of $G$ which contains as its unique maximal compact subgroup the group $U_\AA:=\AA^\times$. This group $U_\AA$ admits a filtration by compact open subgroups, given by $U_\AA^k=1+\PP^k$, for $k\geq 1$. Each $U_\AA^k$ is normalized by $\KK_\AA$.\\

Via $\AA$, we may put a valuation on $A$ by setting $v_\AA(x)=\mathrm{max}\{n\in\Z\ | \ x\in\PP^n\}$, where we take $v_\AA(0)=\infty$.\\

A \emph{stratum} in $A$ is a quadruple $[\AA,n,r,\beta]$ consisting of a hereditary $\dedekind$-order $\AA$, integers $n>r\geq 0$, and $\beta\in A$ an element such that $v_\AA(\beta)\geq -n$. Such a stratum defines a character $\psi_\beta$ of $U_\AA^{r+1}/U_\AA^{n+1}$ by $\psi_\beta(x)=\psi\circ\mathrm{tr}(\beta(x-1))$. We say that two strata $[\AA,n,r,\beta]$ and $[\AA,n',r,\beta']$ are equivalent if the cosets $\beta+\PP_\AA^{-r}$ and $\beta'+\PP_\AA^{-r}$ are equal.\\

We will be specifically interested in certain classes of strata. Say that a stratum $[\AA,n,r,\beta]$ is \emph{pure} if $E:=F[\beta]$ is a field, $E^\times\subset\KK_\AA$ and $n=-v_\AA(\beta)$. We say that a pure stratum is \emph{simple} if it satisfies a further technical condition $r<-k_0(\beta,\AA)$; see \cite[(1.4.5)]{bushnell1993admissible}.\\

Given a simple stratum $[\AA,n,r,\beta]$, we may consider $V$ as an $E$-vector space. This leads to an $E$-algebra $B_\beta=\End_E(V)$ and a hereditary $\dedekind_E$-order $\BB_\beta=\AA\cap B_\beta$ in $B_\beta$ with Jacobson radical $\QQ_\beta=\PP\cap B_\beta$. As before, we may consider $KK_{\BB_\beta},U_{\BB_\beta}$ and $U_{\BB_\beta}^k$. One then has $U_{\BB_\beta}=U_\AA\cap B_\beta$, $U_{\BB_\beta}^k=U_\AA^k\cap B_\beta$ and $e_{\BB_\beta/\dedekind_E}e(E/F)=e_{\AA/\dedekind_F}$.

\subsection{Tame corestriction}

Let $E/F$ be a field extension contained in $A$, with $B=\End_E(V)$. A \emph{tame corestriction} on $A$ relative to $E/F$ is a $(B,B)$-bimodule homomorphism $s:A\rightarrow B$ such that $s(\AA)=\AA\cap\BB$ for every hereditary $\dedekind$-order $\AA$ in $A$ such that $E^\times\subset\KK_\AA$. If $\psi_E$ is an additive character of $E$ of level 1, then there exists a unique tame corestriction $s:A\rightarrow B$ relative to $E/F$ satisfying $\psi\circ\mathrm{tr}_{A/F}(ab)=\psi_E\circ\mathrm{tr}_{B/E}(s(a)b)$ for all $a\in A$ and all $b\in B$ \cite[(1.3.4)]{bushnell1993admissible}.\\

This allows one to approximate strata over a series of various extensions $E=F[\beta]$. We use the following results, which are the content of \cite[(2.4.1)]{bushnell1993admissible}.\\

Firstly, if $[\AA,n,r,\beta]$ is a pure stratum in $\AA$ then there exists a simple stratum $[\AA,n,r,\gamma]$ in $A$ equivalent to $[\AA,n,r,\beta]$. Then if $[\AA,n,r,\beta]$ is pure with $r=-k_0(\beta,\AA)$, then choosing an equivalent simple stratum $[\AA,n,r,\gamma]$ and a tame corestriction $s_\gamma$ on $A$ relative to $F[\gamma]/F$, the pure stratum $[\BB_\gamma,r,r-1,s_\gamma(\beta-\gamma)]$ is equivalent to some simple stratum in $B_\gamma$.

\subsection{Simple characters}

Given a simple stratum $[\AA,n,r,\beta]$, we may define subrings $\HH(\beta,\AA)$ and $\JJ(\beta,\AA)$ of $\AA$; the unit groups $H(\beta,\AA)=\HH(\beta,\AA)^\times$ and $J(\beta,\AA)=\JJ(\beta,\AA)^\times$ of these rings are compact open subgroups of $U_\AA$. The definitions are rather complicated; see \cite[(3.1.14)]{bushnell1993admissible}. Moreover, $H(\beta,\AA)$ is a compact open normal subgroup of $J(\beta,\AA)$, and if we write for each $m\geq 1$ $H^m(\beta,\AA)=H(\beta,\AA)\cap U_\AA^m$ and $J^m(\beta,\AA)=J(\beta,\AA)\cap U_\AA^m$, then we have a chain of inclusions of normal, compact open subgroups
\[H^m(\beta,\AA)\subset J^m(\beta,\AA)\subset J^{m-1}(\beta,\AA).
\]
These groups $H^m(\beta,\AA)$ admit a rather special class of characters known as \emph{simple characters}. Again, the definitions are technical; see \cite[(3.2)]{bushnell1993admissible}. We simple note that, for each simple stratum $[\AA,n,r,\beta]$ and each integer $m\geq 0$, one obtains a set $\Cc(\AA,m,\beta)$ of simple characters of $H^{m+1}(\beta,\AA)$, satisfying a number of desirable properties. Key among these is the ``intertwining implies conjugacy'' property: if $\theta\in\Cc(\AA,m,\beta)$ and $\theta'\in\Cc(\AA,m',\beta')$ are such that $I_G(\theta,\theta')\neq\emptyset$, then $m=m'$ and there exists a $g\in G$ such that $\Cc(\AA',m,\beta')=\Cc({}^g\AA,m,{}^g\beta)$ and $\theta'={}^g\theta$ \cite[(3.5.11)]{bushnell1993admissible}.\\

Of particular interest to us is the case that $m=0$. Here, we have the following:

\begin{theorem}[\cite{bushnell2013intertwining}]
Let $\pi$ be a supercuspidal representation of $G$. Then there exists a simple stratum $[\AA,n,0,\beta]$ and a simple character $\theta\in\Cc(\AA,0,\beta)$ such that $\pi$ contains $\theta$. The simple character $\theta$ is uniquely determined up to $G$-conjugacy.
\end{theorem}

\subsection{Essentially tame supercuspidal representations of $G$}

\begin{definition}
Let $\pi$ be a supercuspidal representation of $G$, containing a maximal simple type $(J=J(\beta,\AA),\lambda)$ corresponding to the simple stratum $[\AA,n,0,\beta]$. We say that $\pi$ is \emph{essentially tame} if $e_\AA$ is coprime to $N$.
\end{definition}

Note that this is well-defined, by the intertwining implies conjugacy property. The main property of essentially tame supercuspidal representations which we will require is that their conjugacy classes of simple characters are rather well-behaved:

\begin{proposition}
Let $[\AA,n,0,\beta]$ be a simple stratum and let $\theta\in\Cc(\AA,0,\beta)$. Suppose that $e_\AA$ is coprime to $N$. Let $[\AA',n',0,\beta']$ be another simple stratum with $H^1(\beta,\AA)=H^1(\beta',\AA')$, let $\theta'\in\Cc(\AA',0,\beta')$, and suppose that there exists a $g\in G$ such that $^g\theta=\theta'$. Then $\Cc(\AA,0,\beta)=\Cc(\AA,0,\beta')$ and $\theta=\theta'$.
\end{proposition}

\subsection{Simple types in $G$}

We now consider those representations of $J(\beta,\AA)$ which contain $\theta\in\Cc(\AA,0,\beta)$. We approach this problem in several stages. Fix a simple stratum $[\AA,n,0,\beta]$ and a simple character $\theta\in\Cc(\AA,0,\beta)$. There exists a \emph{Heisenberg extension} $\eta$ of $\theta$: this is the unique irreducible representation $\eta$ of $J^1(\beta,\AA)$ which contains $\theta$ upon restriction to $H^1(\beta,\AA)$; in fact, $\eta$ restricts to $H^1(\beta,\AA)$ as a sum of copies of $\theta$ \cite[(5.1.1)]{bushnell1993admissible}.\\

Next, we say that a $\beta$-extension of $\eta$ is an extension of $\eta$ to $J^1(\beta,\AA)$ which is intertwined by $B_\beta^\times$. By \cite[(5.2.2)]{bushnell1993admissible}, there always exists a $\beta$-extension of $\eta$.\\

Every irreducible representation of $J(\beta,\AA)$ containing $\theta$ is then of the form $\kappa\otimes\sigma$ for some irreducible representation $\sigma$ of $J(\beta,\AA)/J^1(\beta,\AA)\simeq\prod_{i=1}^{e_{\BB_\beta}/\dedekind_E}\mathbf{GL}_{N/[E:F]}(\kkk_E)$. This brings us to the main definition:

\begin{definition}
A \emph{simple type} in $G$ is a pair $(J,\lambda)$ consisting of a compact open subgroup $J$ of $G$ and an irreducible representation $\lambda$ of $J$, of one of the following two forms:
\begin{enumerate}[(i)]
\item $J=J(\beta,\AA)$ for some simple stratum $[\AA,n,0,\beta]$ and $\lambda=\kappa\otimes\sigma$, where $\kappa$ is a $\beta$-extension of some simple character $\theta\in\Cc(\AA,0,\beta)$ and $\sigma$ is the inflation to $J(\beta,\AA)$ of $\sigma_0^{e_{\BB_\beta/\dedekind_E}}$, where $\sigma_0$ is an irreducible cuspidal representation of $\mathbf{GL}_{N/[E:F]}(\kkk_E)$; or
\item $J=U_\AA$ for a principal $\dedekind$-order in $A$ and $\sigma$ is the inflation of an irreducible cuspidal representation of $U_\AA/U_\AA^1\simeq\prod_{i=1}^{e_\AA}\mathbf{GL}_f(\kkk_F)$ of the form $\sigma=\sigma_0^{e_\AA}$, where $\sigma_0$ is an irreducible cuspidal representation of $\mathbf{GL}_f(\kkk_F)$.
\end{enumerate}
\end{definition}

In practice, there is no need to distinguish between these two cases: the second is essentially a degenerate case of the first, with $\theta=\mathds{1}$, $E=F$ and $\BB_\beta=\AA$.\\

We will be interested in the \emph{maximal simple types}; these are simple types $(J,\lambda)$ constructed from $[\AA,n,0,\beta]$ such that $\BB_\beta$ is a maximal $\dedekind_E$-order in $B_\beta$, i.e. such that $e_{\BB_\beta/\dedekind_E}=1$. Thus, in this case, $J/J^1\simeq\mathbf{GL}_{N/[E:F]}(\kkk_E)$ and $\sigma=\sigma_0$. The main result of \cite{bushnell1993admissible} is the following:

\begin{theorem}
Let $(J,\lambda)$ be a maximal simple type in $G$.
\begin{enumerate}[(i)]
\item There exists a supercuspidal representation $\pi$ of $G$ with $\Hom_J(\pi|_J,\lambda)\neq 0$, and any irreducible representation $\pi'$ of $G$ with $\Hom_J(\pi'|_J,\lambda)\neq 0$ is of the form $\pi'\simeq\pi\otimes(\chi\circ\det)$, for some $\chi\in\XXX_{\mathrm{nr}}(F)$.
\item The $G$-intertwining of $\lambda$ is equal to $E^\times J$. There exists a unique extension $\Lambda$ of $\lambda$ to $E^\times J$ such that $\pi\simeq\cInd_J^G\ \Lambda$.
\end{enumerate}
Conversely, any supercuspidal representation of $G$ contains some maximal simple type, and if $(J,\lambda)$ and $(J',\lambda')$ are two maximal simple types in $G$ then $I_G(\lambda,\lambda')\neq\emptyset$ if and only if there exists a $g\in G$ such that $\lambda'\simeq{}^g\lambda$.
\end{theorem}

Given a maximal simple type $(J,\lambda)$ with $\lambda=\kappa\otimes\sigma$, then there is a convenient way of recovering the representation $\sigma$. Given an irreducible representation $\rho$ of a group $H$ containing $J^1=J^1(\beta,\AA)$, the space $\Hom_{J^1}(\kappa,\rho)$ carries a natural $J$-action given by $j\cdot f=\rho(j)\circ f\circ\kappa(j)^{-1}$, for $j\in J$ and $f\in\Hom_{J^1}(\kappa,\rho)$. Since $f$ is $J^1$-equivariant, this action is trivial on $J^1$ and so this defines a functor $\KKK_\kappa:\Rep(H)\rightarrow\Rep(J/J^1)$ given by $\Hom_{J^1}(\kappa,-)$. This is an exact functor which, in particular, maps admissible representations of $G$ to finite-dimensional representations of $J/J^1$. Given a simple type $(J,\lambda=\kappa\otimes\sigma)$, one has $\KKK_\kappa(\lambda)=\sigma$.

\subsection{Simple types in $\bar{G}$}

We now describe the passage, via Clifford theory, from maximal simple types in $G$ to the corresponding objects in $\bar{G}$. The results in this section are established in \cite{bushnell1993sln,bushnell1994sln}.\\

Let $\pi$ be a supercuspidal representation of $G$, and let $(J,\lambda)$ be a maximal simple type contained in $\pi$. The representation $\pi|_{\bar{G}}$ is, by Clifford theory, isomorphic to a multiplicity-free direct sum of representations which are $G$-conjugate to some supercuspidal representation $\bar{\pi}$ of $\bar{G}$. Every supercuspidal representation of $\bar{G}$ arises in this way.

\begin{definition}
A \emph{maximal simple type} in $\bar{G}$ is a pair of the form $(\bar{J},\mu)$ where $\bar{J}=J\cap\bar{G}$ and $\mu$ is an irreducible subrepresentation of $\lambda|_{\bar{J}}$ for some maximal simple type $(J,\lambda)$ in $G$.
\end{definition}

\begin{theorem}[\cite{bushnell1994sln}]
Let $\pi$ be a supercuspidal representation of $G$, and let $\bar{\pi}$ be an irreducible subrepresentation of $\pi|_{\bar{G}}$. Suppose that $\pi$ contains the maximal simple type $(J,\lambda)$. Then there exists an irreducible subrepresentation $\mu$ of $\lambda|_{\bar{J}}$ such that $\bar{\pi}$ contains the maximal simple type $(\bar{J},\mu)$.\\

Conversely, given a maximal simple type $(\bar{J},\mu)$ in $\bar{G}$, the representation $\cInd_{\bar{J}}^{\bar{G}}\ \mu$ is irreducible and supercuspidal. If $(\bar{J}',\mu')$ is another maximal simple type in $\bar{G}$ such that $I_{\bar{G}}(\mu,\mu')\neq\emptyset$, then there exists a $g\in\bar{G}$ such that $\bar{J}'={}^g\bar{J}$ and $\mu'\simeq{}^g\mu$.
\end{theorem}

\subsection{Types}

We now interpret the constructions of the two preceding sections in a slightly more general context.

\begin{definition}
Let $\pi$ be a supercuspidal representation of a $p$-adic group $\Gg$. A \emph{$[\Gg,\pi]_\Gg$-type} is a pair $(J,\lambda)$ consisting of a compact open subgroup $J$ of $G$ and an irreducible representation $\lambda$ of $J$ such that, for any irreducible representation $\pi'$ of $G$, one has that $\Hom_J(\pi'|_J,\lambda)\neq 0$ if and only if there exists an unramified character $\omega$ of $\Gg$ such that $\pi'\simeq\pi\otimes\omega$.
\end{definition}

In the case that $\Gg=\bar{G}$, the only unramified characters of $G$ are of the form $\omega=\chi\circ\det$ for $\chi\in\XXX_{\mathrm{nr}}(F)$. In the case that $\Gg=\bar{G}$, there are no non-trivial unramified characters, and so the condition simply becomes $\pi'\simeq\pi$. From this, it is simple to check that the maximal simple types discussed above are $[\Gg,\pi]_\Gg$-types for the appropriate choices of $\Gg$ and $\pi$.\\

While we do not go into the details here, we note that this definition makes sense due to more theoretical reasons: a $[G,\pi]_G$-type is a means of describing the block containing $\pi$ in the Bernstein decomposition of $\Rep(G)$ in terms of a finite-dimensional representation of a compact group; see \cite{bushnell1998structure}.\\

In this paper, we will completely classify $[\bar{G},\bar{\pi}]_{\bar{G}}$-types when $\bar{\pi}$ is an essentially tame supercuspidal representation of $\bar{G}$. The above notion of a type turns out to be inconvenient for these purposes. Indeed, from a $[\Gg,\pi]_\Gg$-type $(J,\lambda)$, there are two simple ways of producing new types: forming the pair $({}^gJ,{}^g\lambda)$ for some $g\in\Gg$; or forming the pair $(K,\tau)$, where $K\supset J$ is compact open and $\tau$ is an irreducible subrepresentation of $\Ind_J^K\ \lambda$. We therefore make the following modified definition:

\begin{definition}
A \emph{$[\Gg,\pi]_\Gg$-archetype} is a $\Gg$-conjugacy class of $[\Gg,\pi]_\Gg$-types $(K,\tau)$ with $K\subset G$ a \emph{maximal} compact subgroup.
\end{definition}

It is these archetypes which are amenable to a clean classification. We will often abuse notation, and speak of an archetype $(K,\tau)$ as being a conjuacy class of types, \emph{together with the fixed choice of representative $(K,\tau)$}.

\section{The main results}

Our goal is to show that, given an essentially tame supercuspidal representation $\bar{\pi}$ of $\bar{G}$, any $[\bar{G},\bar{\pi}]_{\bar{G}}$-type which is defined on a maximal compact subgroup of $\bar{G}$ must be induced from a maximal simple type contained in $\bar{\pi}$. The key to this result is the following:

\begin{theorem}\label{extension-to-K}
Let $\pi$ be an essentially tame supercuspidal representation of $G$, and let $\bar{\pi}$ be an irreducible subrepresentation of $\pi|_K$. Suppose that there exists a $[\bar{G},\bar{\pi}]_{\bar{G}}$-type of the form $(\bar{K},\bar{\tau})$. Then there exists an irreducible subrepresentation $\tau$ of $\Ind_{\bar{K}}^K\ \bar{\tau}$ such that $(K,\tau)$ is a $[G,\pi]_G$-type.
\end{theorem}

This is the main technical result of the paper; we delay its proof until section \ref{proof-section} in order to first discuss its consequences.\\

Any essentially tame supercuspidal representation $\bar{\pi}$ of $\bar{G}$ is obtained as a subrepresentation of $\pi|_{\bar{G}}$ for some $\pi$, and by \cite{paskunas2005unicity} we know that any $[G,\pi]_G$-type of the form $(K,\tau)$ for some maximal compact subgroup $K$ of $G$ must be of the form $\tau\simeq\Ind_J^K\ \lambda$ for some maximal simple type $(J,\lambda)$ in $G$. By Frobenius reciprocity, we therefore realize $\bar{\tau}$ as a subrepresentation of
\[\Res_{\bar{K}}^K\Ind_J^K\ \lambda=\bigoplus_{J\backslash K/\bar{K}}\Ind_{^g\bar{J}}^{\bar{K}}\Res_{^g\bar{J}}^{^gJ}\ ^g\lambda.
\]
Any subrepresentation of this representation is of the form $\Ind_{\bar{J}}^{\bar{K}}\ \mu$ for some maximal simple type $(\bar{J},\mu)$ in $\bar{G}$. We therefore conclude that:

\begin{corollary}
Let $\bar{\pi}$ be an essentially tame supercuspidal representation of $\bar{G}$, and let $(\bar{K},\bar{\tau})$ be a $[\bar{G},\bar{\pi}]_{\bar{G}}$-archetype. Then there exists a maximal simple type $(\bar{J},\mu)$ with $\bar{J}\subset\bar{K}$ such that $\bar{\tau}\simeq\Ind_{\bar{J}}^{\bar{K}}\ \mu$.
\end{corollary}

This brings us to our main theorem:

\begin{theorem}[The unicity of types for essentially tame supercuspidal representations of $\SLN(F)$]\label{SLN-unicity}
Let $\bar{\pi}$ be an essentially tame supercuspidal representation of $\bar{G}$.
\begin{enumerate}[(i)]
\item If $(\bar{K},\bar{\tau})$ is a $[\bar{G},\bar{\pi}]_{\bar{G}}$-archetype, then there exists a maximal simple type $(\bar{J},\mu)$ with $\bar{J}\subset\bar{K}$ such that $\bar{\tau}\simeq\Ind_{\bar{J}}^{\bar{K}}\ \mu$. Moreover, $\bar{\tau}$ is contained in $\bar{\pi}$ with multiplicity one.
\item If $\bar{K}$ is a maximal compact subgroup of $\bar{G}$ then there exists at most one $[\bar{G},\bar{\pi}]_{\bar{G}}$-archetype of the form $(\bar{K},\bar{\tau})$.
\item Any two $[\bar{G},\bar{\pi}]_{\bar{G}}$-archetypes are conjugate (up to isomorphism) by an element of $G$.
\item There number of $[\bar{G},\bar{\pi}]_{\bar{G}}$-archetypes is precisely $e_{\bar{\pi}}$.
\end{enumerate}
\end{theorem}

\begin{proof}
We have already established the first claim in \emph{(i)}. To see that $\bar{\tau}$ is contained in $\bar{\pi}$ with multiplicty one, note that by Frobenius reciprocity the multiplicity with which $\bar{\tau}$ appears in $\bar{\pi}|_{\bar{K}}$ is equal to the multiplicity with which $\mu$ appears. By Frobenius reciprocity and a Mackey decomposition we have
\[\Hom_{\bar{J}}(\mu,\bar{\pi}|_{\bar{J}})=\Hom_{\bar{J}}(\mu,\Res_{\bar{J}}^{\bar{G}}\cInd_{\bar{J}}^{\bar{G}}\ \mu)=\bigoplus_{\bar{J}\backslash \bar{G}/\bar{J}}\Hom_{\bar{J}}(\mu,\Ind_{^g\bar{J}\cap\bar{J}}^{\bar{J}}\Res_{^g\bar{J}\cap\bar{J}}^{^g\bar{J}}\ ^g\mu.
\]
Since $I_{\bar{G}}(\mu)=\bar{J}$, the summands in this latter space are non-zero if and only if $g\in\bar{J}$; hence there exists a unique non-zero summand which is one-dimensional.\\

To see \emph{(ii)}, suppose that $\bar{\tau}$ and $\bar{\tau}'$ are two $[\bar{G},\bar{\pi}]_{\bar{G}}$-archetypes defined on $\bar{K}$. Then there exist maximal simple $[\bar{G},\bar{\pi}]_{\bar{G}}$-types $(\bar{J},\mu)$ and $(\bar{J}',\mu')$ which induce to give $\bar{\tau}$ and $\bar{\tau}'$, respectively. Since these maximal simple types both induce irreducibly to $\bar{\pi}$, they intertwine and hence are conjugate. So the induced representations $\bar{\tau}$ and $\bar{\tau}'$ are conjugate by an element of the normalizer of $\bar{K}$, which is simply $\bar{K}$; it follows that $\bar{\tau}\simeq\bar{\tau}'$.\\

Similarly, if $(\bar{K}',\bar{\tau}')$ is another $[\bar{G},\bar{\pi}]_{\bar{G}}$-archetype, then there exist maximal simple types $(\bar{J},\mu)$ and $(\bar{J}',\mu')$ which induce to give $\bar{\tau}$ and $\bar{\tau}'$. As before, these two types are conjugate and so, without loss of generality we may redefine $\bar{K}'$ so that $\mu=\mu'$. We fix a standard set of representatives of $\bar{G}$-conjugacy classes of maximal compact subgroups of $\bar{G}$ which contain $\bar{J}$. We already have one such group in $\bar{K}$. The maximal simple type $(\bar{J},\mu)$ comes from a simple stratum $[\AA,n,0,\beta]$ with $\bar{K}\supset\bar{U}_\AA$; let $\varpi_E$ be a uniformizer of $E=F[\beta]$. Then $\bar{J}$ is contained in each of the groups $^{\varpi_E^j}\bar{K}$, for $0\leq j\leq e_\AA-1$. We claim that there do not exist any other $\bar{G}$-conjugacy classes of maximal compact subgroups of $\bar{G}$ into which $\bar{J}$ admits a containment.\\

Let $\nu$ denote the matrix with $\nu_{i,i+1}=1$ for $1\leq i\leq N-1$, $\nu_{N,1}=\varpi_F$ and $\nu_{i,j}=0$ otherwise; then $\nu$ is a uniformizer of a degree $N$ totally ramified extension of $F$. The $N$ compact open subgroups $^{\nu^j}\bar{K}$, $0\leq j\leq N-1$ form a system of representatives of the $N$ conjugacy classes of maximal compact subgroups of $\bar{G}$. There exists a choice $\varpi_E$ of uniformizer of $E$ such that $^{\varpi_E^j}\bar{K}\subset{}^{\nu^{Nj/e_\AA}}\bar{K}$ for each $0\leq j\leq e_\AA-1$. The group $\bar{J}/\bar{J}^1\simeq\mathbf{SL}_{N/[E:F]}(\kkk_E)$ contains the kernel of the norm map $\NNN_{\kkk_L/\kkk_E}$ on some degree $N/[E:F]$ extension $\kkk_L/\kkk_E$. This kernel is a cyclic group of order $\frac{q^{N/e_\AA}-1}{q-1}$. Suppose that $\bar{J}$ were contained in $^{\nu^k}\bar{K}$ for some value of $k$ other than the $e_\AA$ values constructed above. Then one would have $\bar{J}\subset\left(\bigcap_{i=1}^{e_\AA-1}{}^{\nu^{jN/e_\AA}}\bar{K}\right)\cap{}^{\nu^k}\bar{K}$. This group is equal to $\bar{U}_\CC$ for some hereditary $\dedekind$-order $\CC$ of lattice period $e_\AA+1$ (note that no issue arises if $e_\AA=N$; we have already constructed all possible archetypes).\\

By Zsigmondy's theorem, unless $N/e_\AA=2$ and $q=2^i-1$ or $N/e_\AA=6$ and $q=2$, there exists a prime $r$ dividing $q^{N/e_\AA}-1$ but not dividing $q^s-1$ for any $1\leq s \leq N/e_\AA$. If $N/e_\AA=6$ and $q=2$, let $r=63$, and if $N/e_\AA=2$ and $q=2^i-1$, let $r=4$. While in the latter two cases $r$ is composite, it will be coprime to $q^s-1$ for every $q\leq s\leq N/e_\AA$, which suffices for our purposes. Thus, via the embedding $\ker\NNN_{\kkk_L/\kkk_F}\hookrightarrow\bar{J}/\bar{J}^1$, one obtains in each case an order $r$ element of $\bar{J}\bar{J}^1$, which lifts to give an order $r$ element of $\bar{J}$. The inclusion
\[\bar{J}/\bar{J}^1\hookrightarrow J/J^1\simeq\mathbf{GL}_{N/[E:F]}(\kkk_E)\hookrightarrow\mathbf{GL}_{N/e_\AA}(\kkk_F)\hookrightarrow\prod_{i=1}^{e_\AA}\mathbf{GL}_{N/e_\AA}(\kkk_F),
\]
where the latter map is the diagonal embedding, maps $\bar{J}/\bar{J}^1$ to a block-diagonal group, the blocks of which are pairwise Galois conjugate. So each of the blocks of $\mathbf{GL}_{N/e_\AA}(\kkk_F)$ contains an order $r$ element. However, as one also has $\bar{J}\subset\bar{U}_\CC$, one again obtains an order $r$ element of $U_\CC/U_\CC^1\simeq\prod_{i=1}^{e_\AA+1}\mathbf{GL}_{N_i}(\kkk_F)$, for some partition $N=N_1+\cdots+N_{e_\AA+1}$ of $N$. Among these $N_i$, there will be $e_\AA-1$ which are equal to $N/e_\AA$, and the remaining two are distinct from $N/e_\AA$. Hence in the image of $\ker \NNN_{\kkk_E/\kkk_F}\hookrightarrow U_\CC/U_\CC^1$, one obtains an order $r$ element in a block, which is actually contained in the standard parabolic subgroup of $\mathbf{GL}_{N/e_\AA}(\kkk_F)$ corresponding to the Levi subgroup $\mathbf{GL}_{N_l}(\kkk_F)\times\mathbf{GL}_{N_k}(\kkk_F)$, for some $l+k=N/e_\AA$. But the order of this group is $\left(\prod_{i=0}^{N_l-1}q^i(q^{N_l-i}-1)\right)\cdot\left(\prod_{j=0}^{N_k-1}q^i(q^{N_k-1}-1)\right)$. So $r$ must divide one of these factors. Clearly $r$ cannot divide $q^t$ for any $t$; otherwise $r$ could not divide $q^{N/e_\AA}-1$. Also, as $N_l-i,N_k-i<N/e_\AA$ for all relevant $i$, our choice of $r$ guarantees that $r$ may not divide $|\mathbf{GL}_{N_l}(\kkk_F)\times\mathbf{GL}_{N_k}(\kkk_F)|$. This gives the desired contradiction, and so we conclude that $\bar{J}$ only admits a containment into the $e_\AA$ conjugacy classes of maximal compact subgroups of $\bar{G}$ which were constructed above. This proves \emph{(iv)}.\\

Finally, to see \emph{(iii)}, note that we have already shown that given any two $[\bar{G},\bar{\pi}]_{\bar{G}}$-archetypes of the form $(\bar{K},\bar{\tau})$ and $({}^{\varpi_E^j}\bar{K},\bar{\tau}')$ (we have seen that it is no loss of generality to take our archetypes to be of this form), there exists a maximal simple type $(\bar{J},\mu)$ arising from the simple srtatum $[\AA,n,0,\beta]$ with $\bar{J}\subset\bar{K}\cap{}^{\varpi_E^j}\bar{K}$ such that $\bar{\tau}\simeq\Ind_{\bar{J}}^{\bar{K}}\ \mu$ and $\bar{\tau}'\simeq\Ind_{\bar{J}}^{^{\varpi_E^j}\bar{K}}\ \mu$. Since $\varpi_E^j$ normalizes $\mu$ for any $j$, it follows that $\bar{\tau}'\simeq{}^{\varpi_E^j}\bar{\tau}$.
\end{proof}

\section{Proof of Theorem \ref{extension-to-K}}\label{proof-section}

It remains for us to prove Theorem \ref{extension-to-K}. Let us begin by fixing some notation, on top of that retained from the statement of the theorem. Let $[\AA,n,0,\beta]$ be a simple stratum, and let $\theta\in\Cc(\AA,0,\beta)$ be such that $\pi$ contains $\theta$. Let $\kappa$ be a fixed $\beta$-extension of $\kappa$, and suppose that $\pi$ contains the maximal simple type $\lambda=\kappa\otimes\sigma$ defined on $J=J^0(\beta,\AA)$. As usual, denote by $E$ the field extension $F[\beta]/F$, by $B$ the algebra $\End_E(V)$, and by $\BB$ the hereditary $\dedekind_E$-order $\AA\cap B$. Without loss of generality, we may assume that $J\subset U_\AA\subset K$.

\subsection{First approximation}

We begin by taking the na\"{i}ve approach, and attacking the problem via Clifford theory. This allows us to show that the representation $\Ind_{\bar{K}}^K\ \bar{\tau}$ contains only irreducible subrepresentations which are, in some sense, rather close to being types.\\

We fix, once and for all, an irreducible subrepresentation $\Psi$ of $\Ind_{\bar{K}}^K\ \bar{\tau}$ such that $\Psi$ is contained in $\pi$. Note that such a $\Psi$ clearly exists: by Frobenius reciprocity we have
\[\Hom_K(\Ind_{\bar{K}}^K\ \bar{\tau},\Res_K^G\ \pi)=\Hom_{\bar{K}}(\bar{\tau},\Res_{\bar{K}}^{\bar{G}}\Res_{\bar{G}}^G\ \pi)\neq 0,
\]
and so some irreducible subrepresentation of $\Ind_{\bar{K}}^K\ \bar{\tau}$ is contained in $\pi$.

\begin{lemma}\label{order-N}
Suppose that $\pi'$ is an irreducible representation of $G$ which contains $\Psi$. Then there exists a $\chi\in\XXX_N(F)$ such that $\pi'\simeq\pi\otimes(\chi\circ\det)$.
\end{lemma}

\begin{proof}
We have $\Hom_K(\pi'|_K,\Psi)\neq 0$, and so
\[0\neq\Hom_K(\Ind_{\bar{K}}^K\ \bar{\tau},\Res_K^G\ \pi')=\Hom_{\bar{K}}(\bar{\tau},\Res_{\bar{K}}^{\bar{G}}\Res_{\bar{G}}^G\ \pi').
\]
Hence $\pi'$ must contain $\bar{\pi}$ upon restriction to $\bar{G}$ and so, in particular, $\pi'$ must be a supercuspidal representation of the form $\pi'\simeq\pi\otimes(\chi\circ\det)$ for some character $\chi$ of $F^\times$. Comparing central characters, we see that $\omega_{\pi'}|_{\dedekind^\times}=\omega_\pi|_{\dedekind^\times}$.  We have $\omega_{\pi'}=\omega_\pi\otimes(\chi\circ\det)$, and so $\chi\circ\det$ is trivial on $\dedekind^\times$, which is to say that $\chi^N$ is unramified.
\end{proof}

\subsection{Decompositions of $\pi|_K$}

Let $\Lambda$ be an extension of $\lambda$ to $E^\times J$ such that $\cInd_{E^\times J}^G\ \Lambda\simeq\pi$. It will occasionally be convenient for us to work with slight modifications of $\lambda$ and $\Lambda$. Let $\rho=\Ind_J^{U_\AA}\ \lambda$, and $\tilde{\rho}=\cInd_{E^\times J}^{\KK_\AA}\ \Lambda$. It follows from the fact that $I_G(\lambda)=I_G(\Lambda)=E^\times J$ that both $\rho$ and $\tilde{\rho}$ are irreducible, that $\tilde{\rho}$ is an extension of $\rho$, and that $\pi\simeq\cInd_{\KK_\AA}^G\ \tilde{\rho}$. We therefore obtain two decompositions of the representation $\pi|_K$: from the realization $\pi\simeq\cInd_{E^\times J}^G\ \Lambda$ we obtain the decomposition
\begin{equation}
\label{eq:decomp1}\pi|_K=\Res_K^G\cInd_{E^\times J}^G\ \Lambda=\bigoplus_{E^\times J\backslash G/K}\Ind_{^gJ\cap K}^K\Res_{^gJ\cap K}^{^gJ}\ ^g\lambda;
\end{equation}
while from the realization $\pi\simeq\cInd_{\KK_\AA}^G\ \tilde{\rho}$ we obtain the decomposition
\begin{equation}
\label{eq:decomp2}\pi|_K=\Res_K^G\cInd_{\KK_\AA}^G\ \tilde{\rho}=\bigoplus_{\KK_\AA\backslash G/K}\Ind_{^gU_\AA\cap K}^K\Res_{^gU_\AA\cap K}^{^gU_\AA}\ ^g\rho.
\end{equation}
It is decomposition \eqref{eq:decomp1} in which we will be most interested. However, the double coset space $E^\times J\backslash G/K$ is far too complicated for us to work with directly. We therefore approach the problem via decomposition \eqref{eq:decomp2}. Following Paskunas \cite[Lemma 5.3]{paskunas2005unicity}, we fix a system of coset representatives. Namely, any coset $\KK_\AA g'K$ in $\KK_\AA\backslash G/K$ admits a diagonal representative $g=(\varpi^{a_1},\dots,\varpi^{a_N})$ such that, for all $0\leq i\leq e_\AA$, one has $a_{i(N/e_\AA)+1}\geq\cdots\geq a_{(i+1)N/e_\AA}\geq 0$, and one of the following holds:
\begin{enumerate}[(i)]
\item $a_{j(N/e_\AA)+1}\neq a_{(j+1)N/e_\AA}$, for some $0\leq j<e_\AA$; or
\item \begin{enumerate}[(a)]
\item $a_{i(N/e_\AA)+1}=a_{(i+1)N/e_\AA}$ for all $0\leq i<e$;
\item $a_1\geq 2$; and
\item there exists $1\leq j\leq N$ such that $a_k>0$ if $k<j$, and $a_k=0$ if $k\geq j$, for all $1\leq k\leq N$.
\end{enumerate}
\end{enumerate}
For the remainder of the proof, we will always take our coset representative $g$ to be of the above form.

\begin{definition}
Let $g\in G$ be a coset representative of the above form, which is such that $\KK_\AA gK\neq \KK_\AA K$.
\begin{enumerate}[(i)]
\item We say that $g$ is \emph{of type A} if the map $U_\AA\cap{}^{g^{-1}}K\rightarrow U_\AA/U_\AA^1$ is not surjective.
\item We say that $g$ is \emph{of type B} if the map $U_\AA\cap{}^{g^{-1}}K\rightarrow U_\AA/U_\AA^1$ is surjective.
\end{enumerate}
\end{definition}

\begin{remark}
If $e_\AA=1$ then all coset represetatives of this form must be of type A, and if $e_\AA=N$ then all coset representatives of this form must be of type B.
\end{remark}

Given an irreducible subrepresentation $\xi$ of $\pi|_K$, there exists some coset representative $g$ as above such that $\xi\hookrightarrow\Ind_{^gU_\AA\cap K}^K\Res_{^gU_\AA\cap K}^{^gU_\AA}\ ^g\rho$. We say that $\xi$ is a representation of type A (respectively, type B) if $g$ is a coset representative of type A (respectively, type B).\\

In the case that $\KK_\AA gK=\KK_\AA K$, the representation $\Ind_{^gJ\cap K}^K\Res_{^gJ\cap K}^{^gJ}\ ^g\lambda$ is equal to $\Ind_J^K\ \lambda$, which is the unique $[G,\pi]_G$-archetype. We are thus reduced to three possibilities:
\begin{itemize}
\item the representation $\Psi$ is isomorphic to $\Ind_J^K\ \lambda$; or
\item the representation $\Psi$ is of type A; or
\item the representation $\Psi$ is of type B.
\end{itemize}
In each of the latter two cases, we will argue to obtain a contradiction. It follows that $\Psi$ is a $[G,\pi]_G$-type; whence the desired result.

\subsection{Case 1: $\Psi$ is of type A}

In the case that $\Psi$ is of type A, we may exploit the failure of the map $U_\AA\cap{}^{g^{-1}}K\rightarrow U_\AA/U_\AA^1$ to be surjective in order to turn the problem into one regarding the finite group $J/J^1$. Denote by $H$ the image in $J/J^1$ of $J\cap{}^{g^{-1}}K$. The crucial result is the following observation of Pa\v{s}k\={u}nas:

\begin{lemma}[{\cite[Proposition 6.8]{paskunas2005unicity}}]
For every irreducible representation $\xi$ of $\sigma|_H$ there exists an irreducible representation $\sigma'\not\simeq\sigma$ of $J/J^1$ such that $\Hom_H(\sigma'|_H,\xi)\neq 0$.
\end{lemma}

The use of this is as follows. If $\Psi$ is contained in $\Ind_{^gJ\cap K}^K\Res_{^gJ\cap K}^{^gJ}\ ^g(\kappa\otimes\sigma)$, then there exists an irreducible subrepresentation $\xi$ of $\sigma|_{H}$ such that $\Psi$ is contained in $\Ind_{^gJ\cap K}^K\Res_{^gJ\cap K}^{^gJ}\ ^g(\kappa\otimes\xi)$. This latter representation is contained in $\Ind_{^gJ\cap K}^K\Res_{^gJ\cap K}^{^gJ}\ ^g(\kappa\otimes\sigma')$, and hence so is $\Psi$. There are two cases to consider. We first examine the case that $\sigma'$ may be taken to be non-cuspidal.

\begin{lemma}\label{type-A-non-cuspidal}
Suppose that there exists a non-cuspidal irreducible representation $\sigma'$ of $J/J^1$ such that $\Psi$ is an irreducible subrepresentation of $\Ind_{^gJ\cap K}^K\Res_{^gJ\cap K}^{^gJ}\ ^g(\kappa\otimes\sigma')$. Then there exists a non-cuspidal irreducible representation $\pi'$ of $G$ which contains $\Psi$.
\end{lemma}

\begin{proof}
Let $\Sigma$ be any non-cuspidal irreducible representation of $J/J^1$. Restricting to $H^1$, the representation $\kappa\otimes\Sigma$ is isomorphic to a sum of copies of $\theta$, and so any irreducible representation $\pi'$ of $G$ containing $\kappa\otimes\Sigma$ must contain the simple character $\theta$. If such a representation $\pi'$ were supercuspidal, then it would contain some maximal simple type $(J,\lambda')$, with $\lambda'$ containing $\theta$. Since a supercuspidal representation may only contain a single conjugacy class of simple characters, it must be the case that $\lambda'=\kappa\otimes\sigma''$ for some cuspidal representation $\sigma''$ of $J/J^1$. Performing a Mackey decomposition, we obtain
\[\pi'|_J=\bigoplus_{J\backslash G/J}\Ind_{^hJ\cap J}^J\Res_{^hJ\cap J}^{^hJ}\ ^h\lambda'.
\]
Let $h$ be a coset representative such that $\kappa\otimes\Sigma\hookrightarrow\Ind_{^hJ\cap J}^J\Res_{^hJ\cap J}^{^hJ}\ ^h\lambda'$. By Frobenius reciprocity, it follows that $h$ intertwines $\kappa\otimes\Sigma$ and $\kappa\otimes\sigma''$. By \cite[Proposition 5.3.2]{bushnell1993admissible}, this intertwining set is contained in $J\cdot I_{B^\times}(\Sigma|_{U_\BB},\sigma''|_{U_\BB})\cdot J$; but by \cite[Proposition 6.15]{paskunas2005unicity} the intertwining set $I_{B^\times}(\Sigma|_{U_\BB},\sigma''|_{U_\BB})$ is empty, giving a contradiction. So we conclude that any irreducible representation $\pi'$ of $G$ containing a representation of $J$ of the form $\kappa\otimes\Sigma$ with $\Sigma$ a non-cuspidal representation of $J/J^1$ must itself be non-cuspidal.\\

We now return to the situation at hand, where $\sigma'$ is a non-cuspidal irreducible representation of $J/J^1$ such that $\Psi$ is an irreducible subrepresentation of $\Ind_{^gJ\cap K}^K\Res_{^gJ\cap K}^{^gJ}\ ^g(\kappa\otimes\sigma')\hookrightarrow\Res_K^G\cInd_J^G\ (\kappa\otimes\sigma')$. Since $\Psi$ is irreducible, it is generated by a single vector $w$, say. Consider the subrepresentation $W$ of $\cInd_J^G\ (\kappa\otimes\sigma')$ generated by $w$. This representation $W$ admits an irreducible quotient $\pi'$ which contains $\Psi$. Hence it suffices for us to show that such an irreducible quotient of $W$ contains a representation of the form $\kappa\otimes\Sigma$, with $\Sigma$ a non-cuspidal irreducible representation of $J/J^1$. As $W\hookrightarrow\cInd_J^G\ (\kappa\otimes\sigma')$, we have
\[W|_J\hookrightarrow\Res_J^G\cInd_J^G\ (\kappa\otimes\sigma')=\bigoplus_{J\backslash G/J}\Ind_{^hJ\cap J}^J\Res_{^hJ\cap J}^{^hJ}\ ^h(\kappa\otimes\sigma').
\]
Since $W$ contains $\theta$ upon restriction to $H^1$, it follows that $W\in\Rep^{\SS_\theta}(G)$, as $W$ is generated by a single vector. So any irreducible quotient of $W$ contains $\kappa\otimes\sigma''$ for some irreducible representation $\sigma''$ of $J/J^1$, and it remains to show that $\sigma''$ may not be cuspidal. If $\sigma''$ were cuspidal then we would obtain an inclusion $\kappa\otimes\sigma''\hookrightarrow\Ind_{^hJ\cap J}^J\Res_{^hJ\cap J}^{^hJ}\ ^h(\kappa\otimes\sigma')$ for some $h\in G$, which is to say that $I_G(\kappa\otimes\sigma',\kappa\otimes\sigma'')$. Applying \cite[Proposition 6.16]{paskunas2005unicity}, we see that this intertwining set may only be non-empty if $\sigma''$ is non-cuspidal.
\end{proof}

But we know by Lemma \ref{order-N} that $\Psi$ may only be contained in supercuspidal representations of $G$, leading us to a contradiction if $\sigma'$ may be taken to be non-cuspidal. So it remains to consider the case that any such representation $\sigma'$ must be cuspidal.

\begin{lemma}\label{type-A-dichotomy}
Suppose that the representation $\Psi$ is of type A, corresponding to the coset representative $g$. Then one of the following two statements is true:
\begin{enumerate}[(i)]
\item There exists a non-cuspidal irreducible representation $\pi'$ of $G$ which contains $\Psi$; or
\item The only irreducible representations $\sigma'$ of $J/J^1$ such that $\Psi$ is contained in $\Ind_{^gJ\cap K}^K\Res_{^gJ\cap K}^{^gJ}\ ^g(\kappa\otimes\sigma)$ are of the form $\sigma'\simeq\sigma\otimes(\chi\circ\det)$, for some $\chi\in\XXX_N(F)$ which is trivial on $\det J^1$.
\end{enumerate}
\end{lemma}

\begin{proof}
Let $\sigma'$ be an irreducible representation of $J/J^1$ such that $\Psi$ is contained in $\Ind_{^gJ\cap K}^K\Res_{^gJ\cap K}^{^gJ}\ ^g(\kappa\otimes\sigma)$. If $\sigma'$ may be taken to be non-cuspidal then \emph{(i)} follows from Lemma \ref{type-A-non-cuspidal}. So suppose that $\sigma'$ is cuspidal. Then $\lambda'=\kappa\otimes\sigma'$ is a maximal simple type. Since $\Psi$ is a subrepresentation of 
\[\Ind_{^gJ\cap K}^K\Res_{^gJ\cap K}^{^gJ}\ ^g\lambda'\hookrightarrow\Res_K^G\cInd_J^G\ \lambda',
\]
it follows that $\Psi$ is contained in some irreducible (necessarily supercuspidal) subquotient $\pi'$ of $\cInd_J^G\ \lambda'$, and so $\pi'\simeq\pi\otimes(\chi\circ\det)$ for some $\chi\in\XXX_N(F)$. The representation $\pi$ contains the simple characters $\theta$ and $\theta(\chi\circ\det)$; hence $\theta$ is conjugate to $\theta(\chi\circ\det)$ and, since $e_\AA$ is coprime to $N$, this implies that $\theta=\theta(\chi\circ\det)$, hence $\chi$ is trivial on $\det H^1=\det J^1$.\\

Since $\pi'$ contains a unique archetype, the two representations $\Ind_J^K\ \lambda'$ and $\Ind_J^K\ \lambda\otimes(\chi\circ\det)$ must be isomorphic, and so $I_K(\lambda',\lambda\otimes(\chi\circ\det))\neq\emptyset$. On the other hand, since both $\sigma'$ and $\sigma\otimes(\chi\circ\det)$ are trivial on $J^1$, both $\lambda'$ and $\lambda\otimes(\chi\circ\det)$ become isomorphic to a sum of copies of $\kappa$ upon restriction to $J^1$, and so by \cite[Proposition 5.3.2]{bushnell1993admissible} we see that
\[I_K(\lambda',\lambda\otimes(\chi\circ\det))\subset K\cap J\cdot I_{B^\times}(\lambda'|_{U_{\BB}},\lambda\otimes(\chi\circ\det)|_{U_\BB})\cdot J\subset K\cap J\cdot B^\times \cdot J=J.
\]
So $\lambda'$ and $\lambda\otimes(\chi\circ\det)$ are intertwined by an element of $J$, from which we see that $\lambda'\simeq\lambda\otimes(\chi\circ\det)$. Applying the functor $\Hom_{J^1}(\kappa,-)$, we conclude that $\sigma'\simeq\sigma\otimes(\chi\circ\det)$, which proves \emph{(ii)}.
\end{proof}

This will enable us to perform a simple counting argument in order to show that $\Psi$ may not be of type A. Before completing this argument, we first consider the type B case.

\subsection{Case 2: $\Psi$ is of type B}

In the case that $\Psi$ is of type B we require a different approach, for which we must differentiate between two cases.\\

Suppose first that $k_0(\beta,\AA)\neq -1$. Then $H^1(\beta,\AA)=U_{\BB_\beta}^1H^2(\beta,\AA)$ and so we may view a non-trivial character $\mu$ of $(1+\frak{p}_E)/(1+\frak{p}_E^2)$ as a character of $H^1/H^2$ via the composition
\[\xymatrix{
H^1/H^2\ar[r]^-\sim & U_\BB^1/U_\BB^2\ar[r]^-{\det_B} & (1+\frak{p}_E)/(1+\frak{p}_E^2)\ar[r]^-\mu & \C^\times.
}\]

On the other hand, if $k_0(\beta,\AA)=-1$ then the above approach no longer works. Instead, let $[\AA,n,1,\gamma]$ be a simple stratum equivalent to the pure stratum $[\AA,n,1,\beta]$. Then $\theta\psi_{\beta-\gamma}^{-1}$ is a simple character in $\Cc(\AA,0,\gamma)$.\\

To combine these two cases, we let $\mu$ be as above if $k_0(\beta,\AA)\neq -1$, and let $\mu=\phi_{\beta-\gamma}^{-1}$ otherwise. As noted by Pa\v{s}k\={u}nas during the proofs of \cite[Propositions 7.3,7.16]{paskunas2005unicity}, in each of these two cases we have $\theta\mu=\theta$ on $H^1\cap{}^{g^{-1}}K$. Moreover, in each case $\mu$ is trivial on $H^2$.

\begin{lemma}
The representation $\Psi$ cannot be of type B.
\end{lemma}

\begin{proof}
Since $\Psi$ is an irreducible subrepresentation of $\Ind_{^gH^1\cap K}^K\ ^g\theta|_{^gH^1\cap K}$ and $\theta\mu=\theta$ on $H^1\cap{}^{g^{-1}}K$, we see that $\Psi$ is also a subrepresentation of $\Ind_{^gH^1\cap K}^K\ ^g(\theta\mu)|_{^gH^1\cap K}$; this latter representation is in turn a subrepresentation of $\Res_K^G\cInd_{H^1}^G\ \theta\mu$. Since any irreducible subquotient of $\cInd_K^G\ \Psi$ is supercuspidal representation of the form $\pi\otimes(\chi\circ\det)$ for some $\chi\in\XXX_N(F)$ by Lemma \ref{order-N}, there exists a supercuspidal representation of this form which contains $\theta\mu$. As a supercuspidal representation contains a unique conjugacy class of simple characters, we see that $\theta(\chi\circ\det)$ is conjugate to $\theta\mu$.\\

If $\chi$ is trivial on $\det H^1$ then $\theta$ is conjugate to $\theta\mu$, which is shown to be impossible during the proof s of \cite[Propositions 7.3, 7.16]{paskunas2005unicity}. So $\chi$ is non-trivial on $\det H^1$. Since $g$ is a type B coset representative, we must have $e_\AA>1$; hence $\chi$ is also non-trivial on $\det H^2$. But since $\mu$ is trivial on $H^2$ we see that $\theta|_{H^2}$ is conjugate to $\theta(\chi\circ\det)|_{H^2}$. As $e_\AA$ is coprime to $N$ these two characters must actually be equal, implying that $\chi$ is trivial on $\det H^2$; this is a contradiction.
\end{proof}

\subsection{Conclusion}

We have seen that $\Psi$ may not be of type B. So suppose for contradiction that $\Psi$ is of type A. By Lemma \ref{type-A-dichotomy}, there are two possibilities. If $\Psi$ is contained in a non-cuspidal irreducible representation of $G$, we immediately obtain a contradiction to Lemma \ref{order-N}. So suppose that the only irreducible representations $\sigma'$ of $J/J^1$ for which $\Psi$ is contained in $\Ind_{^gJ\cap K}^K\Res_{^gJ\cap K}^{^gJ}\ ^g(\kappa\otimes\sigma')$ are cuspidal representations of the form $\sigma'\simeq\sigma\otimes(\chi\circ\det)$ for some $\chi\in\XXX_N(F)$ which is trivial on $\det J^1$. There are at most $\mathrm{gcd}(N,q_F-1)$ such characters $\chi$.\\

We first take care of the simple case where the extension $E/F$ is totally ramified. Then, by \cite[Corollary 6.6]{paskunas2005unicity}, the image $H$ in $J/J^1$ of $J\cap{}^{g^{-1}}K$ is contained in some proper parabolic $P$ subgroup of $J/J^1$. Let $P^\mathrm{op}$ denote the parabolic subgroup opposite to $P$, and let $U$ be its unipotent radical. Then the restriction to $U$ of $\Ind_H^{J/J^1}\ \sigma|_H$ surjects onto $\Ind_{H\cap U}^U\Res_{H\cap U}^{J/J^1}\ \sigma$; this latter representation is isomorphic to a sum of copies of the regular representation of $J/J^1$ since $H$ intersects trivially with $U$. Hence there must exist a non-cuspidal representation of $J/J^1$ which identifies with $\sigma$ upon restriction to $H$; this is a contradiction.\\

So we may assume that $E/F$ is not totally ramified. Since any irreducible representation of $J/J^1$ which becomes isomorphic to $\sigma$ upon restriction to $H$ must be isomorphic to $\sigma\otimes(\chi\circ\det)$ for some $\chi\in\XXX_N(F)$ with $\chi$ trivial on $\det J^1$, such a irreducible representation $\sigma'$ also agrees with $\sigma$ upon restriction to $\mathscr{H}=H\cdot\mathbf{SL}_{N/[E:F]}(\kkk_E)\subset J/J^1$. Write $\Xi=\sigma|_{\mathscr{H}}$. Then $\Ind_{\mathscr{H}}^{J/J^1}\ \Xi$ must split as a sum of representations of the form $\sigma\otimes(\chi\circ\det)$, for $\chi\in\XXX_N(F)$ trivial on $\det J^1$. Our claim is that there are at least $q_F$ distinct irreducible subrepresentations of $\Ind_{\mathscr{H}}^{J/J^1}\ \Xi$; this would contradict the fact that there are at most $\mathrm{gcd}(N,q_F-1)<q_F$ such $\chi$, showing that $\Psi$ may not be of type A.\\

Equivalently, we may show that there are at least $q_F$ distinct irreducible characters in $\Ind_{\mathscr{H}}^{J/J^1}\ \mathds{1}$. Since $\mathscr{H}$ contains the derived subgroup of $J/J^1$, the representation $\Ind_{\mathscr{H}}^{J/J^1}\ \mathds{1}$ splits as a multiplicity-free direct sum of $[J/J^1:\mathscr{H}]$ distinct characters of $J/J^1$. So it suffices for us to show that the index of $\mathscr{H}$ in $J/J^1$ is at least $q_F$.\\

As $E/F$ is not totally ramified, $\kkk_E/\kkk_F$ is a non-trivial extension. Then there exists a proper subextension $\kkk$ of $\kkk_E$ which contains $\kkk_F$ and is of maximal degree among such extensions of $\kkk_F$ such that $H$ contains only $\kkk$-rational points of $J/J^1$ (by combining \cite[Lemma 6.5]{paskunas2005unicity} and \cite[Corollary 6.6]{paskunas2005unicity}). Thus, if $f=f(E/F)$ is the residue class degree of $E/F$ then $\kkk\simeq\F_{q_F^{f-1}}$, and so we may certainly take as a lower bound for $[J/J^1:\mathscr{H}]$ the number
\[\frac{|\mathbf{GL}_{N/[E:F]}(\kkk_E)|}{|\mathbf{GL}_{N/[E:F]}(\kkk)\cdot\mathbf{SL}_{N/[E:F]}(\kkk_E)|}=\frac{|\mathbf{GL}_{N/[E:F]}(\kkk_E)|}{|\mathbf{SL}_{N/[E:F]}(\kkk_E)|}\cdot\left(\frac{\mathbf{GL}_{N/[E:F]}(\kkk)}{\mathbf{SL}_{N/[E:F]}(\kkk)}\right)^{-1}=\frac{q_F^f-1}{q_F^{f-1}-1}.
\]
This is no less than $q_F$.\\

So the representation $\Psi$ may not be of type A. We conclude that $\Psi\simeq\Ind_J^K\ \lambda$, completing the proof of Theorem \ref{extension-to-K}.

\section{The inertial correspondence}

\subsection{The local Langlands correspondence}

We now give a Galois theoretic interpretation of our unicity results, via the local Langlands correspondence. This allows us to completely describe the fibres of an inertial form of the local Langlands correspondence for $\bar{G}$.\\

Let $\bar{F}/F$ be a separable algebraic closure of $F$ with absolute Galois group $\Gal(\bar{F}/F)$, and let $W_F\subset\Gal(\bar{F}/F)$ be the Weil group: this is the pre-image of $\Z$ under the canonical map $\Gal(\bar{F}/F)\rightarrow\Gal(\bar{\kkk}/\kkk)\simeq\hat{\Z}$. Let $I_F=\ker(\Gal(\bar{F}/F)\rightarrow\Gal(\bar{\kkk}/\kkk)$ be the inertia group; this is the maximal compact subgroup of $W_F$. Fix a choice $\Phi$ of geometric Frobenius element in $W_F$, i.e. an element which maps to $-1\in\hat{Z}$ under the above projection.\\

Given a $p$-adic group $\Gg$, denote by $\hat{\Gg}$ its Langlands dual group. In particular, if $\Gg=\GLN(F)$ then $\hat{\Gg}=\GLN(\C)$, and if $\Gg=\SLN(F)$ then $\hat{\Gg}=\mathbf{PGL}_N(\C)$.\\

The local Langlands correspondence for $G$ gives a unique natural bijective correspondence $\rec_G:\Irr_\sc(G)\rightarrow\Ll_\sc(G)$ between the set $\Irr_\sc(G)$ of isomorphism classes of supercuspidal representations of $G$ and the set $\Ll_\sc(G)$ of isomorphism classes of of irreducible representations $W_F\rightarrow\mathbf{GL}_N(\C)$ such that the image of $\Phi$ is semisimple \cite{harris2001shimura}. (Of course, the local Langlands correspondence for $G$ is more general than this; however, we will only be interested in such representations).\\

From this, following \cite{labesse1979sl2,gelbart1982sln}, it is possible to deduce the local Langlands correspondence for $\bar{G}$. Denote by $\Irr_\sc(\bar{G})$ the set of isomorphism classes of irreducible subrepresentations $\bar{\pi}$ of $\pi|_{\bar{G}}$, for $\pi\in\Irr_\sc(G)$, and denote by $\Ll_\sc(\bar{G})$ the set of projective representations $W_F\rightarrow\mathbf{PGL}_N(\C)$ which lift to an element of $\Ll_\sc(G)$. Let $R:\Irr_\sc(G)\rightarrow\Irr_\sc(\bar{G})$ be a map which associates to each $\pi$ an irreducible subrepresentation of $\pi|_{\bar{G}}$. Then there exists a unique surjective, finite-to-one map $\rec_{\bar{G}}:\Irr_\sc(\bar{G})\rightarrow\Ll_\sc(\bar{G})$ such that the following diagram commutes \emph{for all such choices of} $R$:
\[\xymatrix{
\Irr_\sc(G)\ar[r]^-{\rec_G}\ar[d]_R & \Ll_\sc(G)\ar[d]\\
\Irr_\sc(\bar{G})\ar[r]_-{\rec_{\bar{G}}} & \Ll_\sc(\bar{G})
}\]
Here, the map $\Ll_\sc(G)\rightarrow\Ll_\sc(\bar{G})$ is given by composition with the natural projection $\GLN(\C)\rightarrow\mathbf{PGL}_N(\C)$.\\

This map $\rec_{\bar{G}}$ is the local Langlands correspondence for (the supercuspidal representations of) $\bar{G}$. Its finite fibres are the $L$\emph{-packets} in $\Irr_\sc(G)$.

\subsection{Types and $L$-packets}

For $\Gg=G$ or $\bar{G}$, denote by $\Irr_\et(\Gg)$ the subset of $\Irr_\sc(\Gg)$ consisting of essentially tame supercuspidal representations, and denote by $\Ll_\et(\Gg)$ its image under $\rec_\Gg$. Let $\Aa_\et(\Gg)$ denote the set of $[\Gg,\pi]_{\Gg}$-archetypes, for $\pi\in\Irr_\et(\Gg)$. We have already completely understood the relationship between $\Aa_\et(G)$ and $\Aa_\et(\bar{G})$; we now reinterpret this understanding in terms of the local Langlands correspondences for $G$ and $\bar{G}$. The first step is to establish a form of converse to Theorem \ref{extension-to-K}.

\begin{proposition}\label{restriction-to-Kbar}
Let $\pi$ be an essentially tame supercuspidal representation of $G$, and let $(K,\tau)$ be the unique $[G,\pi]_G$-archetype. Let $\bar{\pi}$ be an irreducible subrepresentation of $\pi|_{\bar{G}}$. Then there exists a $g\in G$ and an irreducible component $\bar{\tau}$ of $^g\tau|_{^g\bar{K}}$ such that $({}^g\bar{K},\bar{\tau})$ is a $[\bar{G},\bar{\pi}]_{\bar{G}}$-archetype.
\end{proposition}

\begin{proof}
Without loss of generality, assume that $\bar{\pi}\simeq\cInd_{\bar{K}}^{\bar{G}}\ \tilde{\mu}$, where $\tilde{\mu}=\cInd_{\bar{J}}^{\bar{K}}\ \mu$ for some maximal simple type $(\bar{J},\mu)$ (if not, replace $\bar{\pi}$ with a $G$-conjugate for which we may do so; clearly the desired result is true for $\bar{\pi}$ if and only if it is true for every $G$-conjugate of $\bar{\pi}$). Let $\tau|_{\bar{K}}=\bigoplus_j\bar{\tau}_j$. We first show that any $\pi'\in\Irr(\bar{G})$ containing some $\bar{\tau}_j$ must appear in the restriction to $\bar{G}$ of $\pi$. We have a non-zero map in
\[\bigoplus_j\Hom_{\bar{K}}(\bar{\tau}_j,\pi'|_{\bar{K}})=\Hom_{\bar{K}}(\tau|_{\bar{K}},\pi'|_{\bar{K}})=\Hom_{\bar{G}}(\cInd_{\bar{K}}^{\bar{G}}\Res_{\bar{K}}^K\ \tau,\pi'),
\]
and so $\pi'$ is a subquotient of $\Res_{\bar{G}}^G\cInd_K^G\ \tau$. Every irreducible subquotient of $\cInd_K^G\ \tau$ is a twist of $\pi$, and hence coincides with $\pi$ upon restriction to $\bar{G}$, and so any irreducible representation $\pi'$ must be of the required form. Hence the possible representations $\pi'$ all lie in a single $G$-conjugacy class of irreducible representations of $\bar{G}$. Let $g\in G$ be such that $^g\pi'\simeq\bar{\pi}$, hence $\pi'\simeq\cInd_{^g\bar{K}}^{\bar{G}}\ ^g\tilde{\mu}$, and choose $j$ so that $\pi'$ contains $\bar{\tau}_j$. We claim that $({}^g\bar{K},{}^g\bar{\tau}_j)$ is the required type.\\

It suffices to show that any $G$-conjugate of $\bar{\pi}$ containing $({}^g\bar{K},{}^g\bar{\tau}_j)$ is isomorphic to $\bar{\pi}$. Suppose that, for some $h\in G$ we have $\Hom_{^g\bar{K}}({}^g\bar{\pi},{}^g\bar{\tau}_j)\neq 0$. The representation $^h\bar{\pi}$ is of the form $^h\bar{\pi}\simeq\cInd_{^h\bar{J}}^{\bar{G}}\ ^h\mu$, and so $^g\bar{\tau}_j$ is induced from some maximal simple type $(\bar{J}',\mu')$, say. So we have
\begin{align*}
0&\neq\Hom_{^g\bar{K}}(\Res_{^g\bar{K}}^{\bar{G}}\ \bar{\pi},{}^g\bar{\tau}_j)\\
&=\Hom_{\bar{J}'}(\Res_{\bar{J}'}^{\bar{G}}\cInd_{^h\bar{J}}^{\bar{G}}\ ^h\mu,\mu')\\
&=\bigoplus_{^h\bar{J}\backslash\bar{G}/\bar{J}'}\Hom_{\bar{J}'}(\cInd_{^{xh}\bar{J}\cap\bar{J}'}^{\bar{J}'}\Res_{^{xh}\bar{J}\cap\bar{J}'}^{^{xh}\bar{J}}\ ^{xh}\mu,\mu')\\
&=\bigoplus_{^h\bar{J}\backslash\bar{G}/\bar{J}'}\Hom_{^{xh}\bar{J}\cap\bar{J}'}(\Res_{^{xh}\bar{J}\cap\bar{J}'}^{^{xh}\bar{J}}\ ^{xh}\mu,\Res_{^{xh}\bar{J}\cap\bar{J}'}^{\bar{J}'}\ \mu').
\end{align*}
So $^h\mu$ and $\mu'$ intertwine in $\bar{G}$, and are therefore $\bar{G}$-conjugate. Hence $\pi'$ is in fact $\bar{G}$-conjugate to $\bar{\pi}$, i.e. $\bar{\pi}\simeq\pi'$ and the result follows.
\end{proof}

\begin{theorem}\label{L-packets}
Let $\pi$ be an essentially tame supercuspidal representation of $G$, and let $(K,\tau)$ be the unique $[G,\pi]_G$-archetype. Let $\Pi$ be the $L$-packet of irreducible subrepresentations of $\pi|_{\bar{G}}$. Then the set of $[\bar{G},\bar{\pi}]_{\bar{G}}$-archetypes for $\bar{\pi}\in \Pi$ is equal to the set of archetypes of the form $({}^g\bar{K},{}^g\bar{\tau})$ for $g\in G$.
\end{theorem}

\begin{proof}
We show that the union of the sets of $[\bar{G},\bar{\pi}]_{\bar{G}}$-types of the form $(\bar{K},\bar{\tau})$, as $\bar{\pi}$ ranges over $\Pi$, is equal to the set of irreducible subrepresentations of $\tau|_{\bar{K}}$; the general result then follows easily. Let $(\bar{K},\bar{\tau})$ be such an archetype. By Theorem \ref{SLN-unicity}, $\bar{\tau}$ is of the required form. Conversely, the irreducible subrepresentations of $\tau|_K$ are pairwise $K$-conjugate by Clifford theory, and so if one of them is a type for some element of $\Pi$ then they all must be. By Proposition \ref{restriction-to-Kbar}, at least one of them must be a type for some $\bar{\pi}\in\Pi$.
\end{proof}

\subsection{The inertial correspondence}

For $\Gg=G$ or $\bar{G}$, let $\Ii_\et(\Gg)$ denote the set of representations $I_F\rightarrow\hat{\Gg}$ which are of the form $\varphi|_{I_F}$ for some $\varphi\in\Ll_\et(\Gg)$; we call such representations \emph{essentially tame inertial types}. We begin by recalling the inertial Langlands correspondence for $G$:

\begin{theorem}[The inertial local Langlands correspondence for $\GLN(F)$ {\cite[Corollary 8.2]{paskunas2005unicity}}]
There exists a unique bijective map $\iner_G:\Aa_\et(G)\rightarrow\Ii_\et(G)$ such that, if $R$ is the map which assigns to a supercuspidal representation $\pi$ of $G$ the unique $[G,\pi]_G$-archetype representation of $K$, then the following diagram commutes:
\[\xymatrix{
\Irr_\et(G)\ar[r]^-{\rec_G}\ar[d]_R & \Ll_\et(G)\ar[d]^{\Res_{I_F}^{W_F}}\\
\Aa_\et(G)\ar[r]_-{\iner_G} & \Ii_\et(G)
}\]
\end{theorem}

Note that while the statement of \cite[Corollary 8.2]{paskunas2005unicity} is not stated in this language, it is trivial to show that the two statements are equivalent. It is the above form of the statement which admits a reasonable generalization to $\bar{G}$.\\

As a notational convenience, we transfer some notation to the setting of $L$-parameters and inertial types. Given $\varphi\in\Ii_\et(\bar{G})$, let $\tilde{\varphi}\in\Ll_\et(G)$ be a lift of some extension of $\varphi$ to $W_F$. Write $\ell_\varphi=\mathrm{length}(\rec_G^{-1}(\tilde{\varphi})|_{\bar{G}})$, and $e_\varphi$ for the lattice period of the hereditary order $\AA$ such that $\rec^{-1}(\tilde{\varphi}$ contains a simple character in $\Aa(\AA,0,\beta)$ for some $\beta$.\\

We come to our main result:

\begin{theorem}[The essentially tame inertial Langlands correspondence for $\SLN(F)$]\label{SLN-inertial}
There exists a unique surjective map $\iner_{\bar{G}}:\Aa_\et(\bar{G})\twoheadrightarrow\Ii_\et(\bar{G})$ with finite fibres such that, for \emph{any} map $T$ assigning to a supercuspidal representation $\bar{\pi}$ of $\bar{G}$ one of the $[\bar{G},\bar{\pi}]_{\bar{G}}$-archetypes, the following diagram commutes:
\[\xymatrix{
\Irr_\et(\bar{G})\ar[r]^-{\rec_{\bar{G}}}\ar[d]_T & \Ll_\et(\bar{G})\ar[d]^-{\Res_{I_F}^{W_F}}\\
\Aa_\et(\bar{G})\ar[r]_-{\iner_{\bar{G}}} & \Ii_\et(\bar{G})
}\]
Each of the fibres of $\iner_{\bar{G}}$ consists of the full orbit under $G$-conjugacy of an archetype, with the fibre above an inertial type $\varphi$ being of cardinality $e_\varphi\ell_\varphi$.\\

Moreover, for any map $R$ assigning to each $[G,\pi]_G$-archetype a $[\bar{G},\bar{\pi}]_{\bar{G}}$-archetype, for $\bar{\pi}$ an irreducible subquotient of $\pi|_{\bar{G}}$, there is a commutative diagram
\[\xymatrix{
\Aa_\et(G)\ar[r]^-{\iner_G}\ar[d]_R & \Ii_\et(G)\ar[d]\\
\Aa_\et(\bar{G})\ar[r]_-{\iner_{\bar{G}}} & \Ii_\et(\bar{G})
}\]
where the map $\Ii_\et(G)\rightarrow\Ii_\et(\bar{G})$ is given by composition with the projection $\GLN(\C)\rightarrow\mathbf{PGL}_N(\C)$.
\end{theorem} 

\begin{proof}
Let $S$ be any map which assigns to each archetype $(\bar{K},\bar{\tau})$ in $\Aa_\et(\bar{G})$ the irreducible subrepresentation $\bar{\pi}=\cInd_{\bar{K}}^{\bar{G}}\ \bar{\tau}$. Let $\iner_{\bar{G}}$ denote the composition $\Res_{I_F}^{W_F}\circ\rec_{\bar{G}}\circ S$. Let $\varphi\in\Ii_\et(\bar{G})$, and let $\tilde{\varphi}$ be an extension of $\varphi$ to $W_F$. Let $\Pi=\rec^{-1}(\tilde{\varphi})$. Then $\Pi=\{\bar{\pi}_i\}$ is an $L$-packet of supercuspidal representations of $\bar{G}$ consisting of the set of irreducible subrepresentations of some supercuspidal representation $\pi$ of $G$. By Theorem \ref{L-packets}, the finite set $\{(\mathscr{K}_i,\bar{\tau}_i\}$ of $[\bar{G},\bar{\pi}_i]_{\bar{G}}$-archetypes, as $\bar{\pi}_i$ ranges through $\Pi$ is precisely the set of archetypes given by the irreducible subrepresentations of $(^g\bar{K},{}^g\tau|_{{}^g\bar{K}})$, for $g\in G$. As each $\bar{\pi}_i$ is an archetype, it follows that for all irreducible representations $\bar{\pi}$ of $\bar{G}$, we have that $\bar{\pi}$ contains some $\bar{\tau}_i$ upon restriction to $\mathscr{K}_i$ if and only if $\bar{\pi}\in\Pi$, if and only if $\rec(\bar{\pi})|_{I_F}\simeq\varphi$. So the map $\iner_{\bar{G}}$ is well-defined, and is the unique map map making the first diagram commute.\\

We now consider the fibres of $\iner_{\bar{G}}$. Let $\varphi\in\Ii_\et(\bar{G})$. Each of the archetypes in $\iner^{-1}(\varphi)$ is represented by a representation of the form $\bar{\tau}=\Ind_{\bar{J}}^{\bar{K}}\ \mu$, for some maximal simple type $(\bar{J},\mu)$ contained in an essentially tame supercuspidal representation, and some maximal compact subgroup $\bar{K}$ of $\bar{G}$ which contains $\bar{J}$. Moreover, any $G$-conjugate of $(\bar{K},\bar{\tau})$ is also contained in the fibre above $\varphi$. Conversely, we have seen that any two archetypes in the same fibre of $\iner_{\bar{G}}$ are $G$-conjugate.\\

So it remains only to calculate the cardinality of $\iner_{\bar{G}}^{-1}(\varphi)$. Let $\tilde{\varphi}\in\Ll_\et(\bar{G})$ be an extension of $\tilde{\varphi}$, and let $\bar{\pi}$ be contained in the $L$-packet $\rec_{\bar{G}}^{-1}(\tilde{\varphi})$. The cardinality of this $L$-packet is $\mathrm{length}(\pi|_{\bar{G}})$, where $\pi$ is any representation of $G$ such that $\bar{\pi}\hookrightarrow\pi|_{\bar{G}}$, i.e. $\#\rec_{\bar{G}}^{-1}(\tilde{\varphi})=\ell_\varphi$. So the fibre $\iner_{\bar{G}}^{-1}(\varphi)$ is equal to the disjoint union of the sets of archetypes contained in each of the $\ell_\varphi$ elements of $\rec^{-1}(\tilde{\varphi})$. Since any two elements of this $L$-packets are $G$-conjugate, any two elements admit the same number of archetypes, which is $e_\varphi$ by Theorem \ref{SLN-unicity}. So we conclude that $\#\iner_{\bar{G}}^{-1}(\varphi)=e_\varphi\ell_\varphi$.\\

The commutativity of the second diagram is simply a translation of Theorem \ref{L-packets} into the language of the inertial correspondence.
\end{proof}

\bibliographystyle{amsalpha}
\addcontentsline{toc}{section}{References}
\bibliography{Archetypes}

\def\cprime{$'$}
\providecommand{\bysame}{\leavevmode\hbox to3em{\hrulefill}\thinspace}
\providecommand{\MR}{\relax\ifhmode\unskip\space\fi MR }
\providecommand{\MRhref}[2]{%
  \href{http://www.ams.org/mathscinet-getitem?mr=#1}{#2}
}
\providecommand{\href}[2]{#2}
\begin{thebibliography}{BK93b}

\bibitem[BH13]{bushnell2013intertwining}
Colin~J. Bushnell and Guy Henniart, \emph{Intertwining of simple characters in
  {${\rm GL}(n)$}}, Int. Math. Res. Not. IMRN (2013), no.~17, 3977--3987.
  \MR{3096916}

\bibitem[BK93a]{bushnell1993admissible}
Colin~J. Bushnell and Philip~C. Kutzko, \emph{The admissible dual of {${\rm
  GL}(N)$} via compact open subgroups}, Annals of Mathematics Studies, vol.
  129, Princeton University Press, Princeton, NJ, 1993. \MR{1204652
  (94h:22007)}

\bibitem[BK93b]{bushnell1993sln}
\bysame, \emph{The admissible dual of {${\rm SL}(N)$}. {I}}, Ann. Sci. \'Ecole
  Norm. Sup. (4) \textbf{26} (1993), no.~2, 261--280. \MR{1209709 (94a:22033)}

\bibitem[BK94]{bushnell1994sln}
\bysame, \emph{The admissible dual of {${\rm SL}(N)$}. {II}}, Proc. London
  Math. Soc. (3) \textbf{68} (1994), no.~2, 317--379. \MR{1253507 (94k:22035)}

\bibitem[BK98]{bushnell1998structure}
\bysame, \emph{Smooth representations of reductive {$p$}-adic groups: structure
  theory via types}, Proc. London Math. Soc. (3) \textbf{77} (1998), no.~3,
  582--634. \MR{1643417 (2000c:22014)}

\bibitem[BM02]{breuil2002multiplicites}
Christophe Breuil and Ariane M{\'e}zard, \emph{Multiplicit\'es modulaires et
  repr\'esentations de {${\rm GL}_2({\bf Z}_p)$} et de {${\rm
  Gal}(\overline{\bf Q}_p/{\bf Q}_p)$} en {$l=p$}}, Duke Math. J. \textbf{115}
  (2002), no.~2, 205--310, With an appendix by Guy Henniart. \MR{1944572
  (2004i:11052)}

\bibitem[GK82]{gelbart1982sln}
S.~S. Gelbart and A.~W. Knapp, \emph{{$L$}-indistinguishability and {$R$}
  groups for the special linear group}, Adv. in Math. \textbf{43} (1982),
  no.~2, 101--121. \MR{644669 (83j:22009)}

\bibitem[HT01]{harris2001shimura}
Michael Harris and Richard Taylor, \emph{The geometry and cohomology of some
  simple {S}himura varieties}, Annals of Mathematics Studies, vol. 151,
  Princeton University Press, Princeton, NJ, 2001, With an appendix by Vladimir
  G. Berkovich. \MR{1876802 (2002m:11050)}

\bibitem[Lat15]{latham2015sl2}
Peter Latham, \emph{Unicity of types for supercuspidal representations of
  $p$-adic $\mathbf{SL}_2$}, J. Number Theory \textbf{162} (2015), 376--390.

\bibitem[LL79]{labesse1979sl2}
J.-P. Labesse and R.~P. Langlands, \emph{{$L$}-indistinguishability for {${\rm
  SL}(2)$}}, Canad. J. Math. \textbf{31} (1979), no.~4, 726--785. \MR{540902
  (81b:22017)}

\bibitem[Pas05]{paskunas2005unicity}
Vytautas Paskunas, \emph{Unicity of types for supercuspidal representations of
  {${\rm GL}_N$}}, Proc. London Math. Soc. (3) \textbf{91} (2005), no.~3,
  623--654. \MR{2180458 (2007b:22018)}

\end{thebibliography}

\end{document}